\documentclass[12pt]{article}
\usepackage{amsfonts,amsmath,amscd,amsthm}
\usepackage[a4paper,left=3cm,right=2cm,top=3cm,bottom=2cm]{geometry}
\usepackage{hyperref}
\usepackage{lmodern}
\usepackage[utf8x]{inputenx} %
\usepackage[LGR, T1]{fontenc}
\usepackage{textcomp} 
\usepackage{graphicx}
\usepackage{wrapfig}

\newcommand{\veps}{\varepsilon}

\newcommand{\R}{\mathbb{R}}

\newtheorem{remark}{Remark}[section]

\newtheorem{theorem}{Theorem}[section]
\newtheorem{corollary}{Corollary}[section]

\begin{document}

\title{Quasi-Einstein manifolds with structure  of warped product}

\author{
\textbf{Paula Gon\c calves Correia Bonfim}
\\
{\small\it e-mail:  paulacorreiacatu@hotmail.com}
  \\
\textbf{Romildo Pina}
\\
{\small\it IME, Universidade Federal de Goi\'as,}\\
{\small\it Caixa Postal 131, 74001-970, Goi\^ania, GO, Brazil }\\
{\small\it e-mail: romildo@ufg.br }
}

\maketitle

\thispagestyle{empty}

\markboth{abstract}{abstract}
\addcontentsline{toc}{chapter}{abstract}

\begin{abstract}
\noindent

In this paper, we prove that under certain conditions, in a quasi-Einstein semi-Riemannian warped product the fiber is necessarily a Einstein manifold. We provide all the quasi-Einstein manifolds when r-Bakry-Emery tensor is null, the base is conformal to an n-dimensional pseudo-Euclidean space, invariant under the action of an $(n-1)$-dimensional translation group and the fiber is Ricci-flat. As an application, we have built a family of Ricci-flat Einstein warped product whose base is not locally conformally flat.
\end{abstract}


\section{Introduction}

In the early 1980s, in order to solve the Poincaré conjecture, Richard Hamilton proposed the initial value problem
$$\begin{array}{c}
\displaystyle \frac{\partial}{\partial t}g(t)=-2Ric_{g(t)},\\
g(0)=g_0,
\end{array}$$
which represents an evolution for a family of Riemannian metrics in a given differentiable manifold, where $Ric_{g(t)}$ is the Ricci tensor in the metric $g(t)$, $g_0$ is a given initial metric and $t$ is the deformation time. This initial value problem became known as the Ricci Flow (see \cite{Chow-Knopf}). In his studies Hamilton realized that the Einstein manifolds were fixed points of the Ricci Flow, and then these manifolds became the subject of several researches.

In attempting to discover some properties of the Einstein manifolds, disturbances of the definition itself have arisen, for example the concept of quasi-Einstein manifolds. A Riemannian manifold $(M^n,g)$ $n$-dimensional is said quasi-Einstein to satisfy
\begin{equation} \label{eqoriginalquasiEinstein}
Ric_g + Hess_gf - \frac{1}{r}df\otimes df = \rho g,
\end{equation}
with $r\in (0,\infty]$, $\rho\in\mathbb{R}$ and $f\in C^\infty(M)$ (see \cite{Case-Shu-Wei}). The tensor on the left side of equality is the $r$-Bakry-Emery tensor. Note that if $r=\infty$, $(M^n,g,f)$ is a gradient Ricci soliton with potential function $f$. If $f$ is constant, $(M,g)$ is a Einstein manifold. Now, if $r$ is finite, we can consider the function $h=e^{-\frac{f}{r}}$, and the equation ($\ref{eqoriginalquasiEinstein}$) becomes
\begin{equation} \label{eqquasiEinstein}
Ric_g - \frac{r}{h}Hess_g h = \rho g.
\end{equation}
The definition of the quasi-Einstein manifolds is naturally extended to the semi-Riemannian case (see \cite{BV-GR-GF}). In our work we will consider semi-Riemannian manifolds satisfying (\ref{eqquasiEinstein}) with $r\in\mathbb{N}$.

In \cite{Kim-Kim}, the authors proved that a Riemannian warped product $M=B\times_hF$ is Einstein with $Ric_g=\rho g$ if and only if
\begin{equation} \label{eqsdoKimKim}
\begin{array}{c}
\displaystyle Ric_{g_B} = \rho g_B + \frac{r}{h}Hess_{g_B}h,\\
Ric_{g_F}=\mu g_F,\\
h\Delta_{g_B}h + (r-1)|grad_{g_B}h|^2 + \rho h^2 = \mu,
\end{array}
\end{equation}
where $g_B$ and $g_F$ represent the metrics of the base and fiber, respectively. We observe that the first equation of (\ref{eqsdoKimKim}) tells us that $B$ is quasi-Einstein and the second equation shows that $F$ is Einstein with Ricci constant curvature $\mu$ satisfying the third. In the same work, these authors have shown that if a manifold $B$ satisfies the first equation of (\ref{eqsdoKimKim}), then the torsion function $h$ satisfies the third equation for some $\mu \in\mathbb{R}$. Thus, choosing a manifold $F$ of dimension $r$ and with Ricci curvature $\mu$, we can construct a Einstein manifold $B \times_hF $.

In \cite{He-Petersen-Wylie}, the authors classified the Einstein Riemannian warped products when the base is locally conformally flat. Already, in \cite{Marcio-Romildo}, the authors studied in  Einstein semi-Riemannian warped product when the base is conformal to an $n$-dimensional pseudo-Euclidean space.

In \cite{BV-GR-VL}, the authors had already shown necessary and sufficient conditions for a Riemannian product warped $M = B \times_fF$ to be locally conformed flat, and these conditions impose strong restrictions on the base and the fiber.

In \cite{Catino}, the authors have proved that any complete locally conformally flat quasi-Einstein manifold of dimension $n\geq 3$ is locally a warped product with $(n−1)$–dimensional fibers of constant curvature.

In this paper we will generalize the work of Sousa and Pina (see \cite{Marcio-Romildo}) presenting a family of Einstein semi-Riemannian warped product whose base is not locally conformally flat. We will study quasi-Einstein manifolds in order to construct new examples of Einstein manifolds using the result obtained in \cite{Kim-Kim}, which are also valid in the semi-Riemannian case. Due to the conditions in \cite{BV-GR-VL}, we will use warped product metrics in an attempt to obtain quasi-Einstein manifolds that are not locally conformally flat. Thus, in the end we get examples of Einstein semi-Riemannian manifolds with warped product metric, where the base is a quasi-Einstein manifold (warped product) that is not locally conformally flat.

Initially, we proved that if in a non trivial warped product $(B\times_fF,\widetilde{g})$ exists a function $h$ satisfying $Hess_{\widetilde{g}}h(X,Y)=0, \ \forall X\in\mathcal{L}(B),Y\in\mathcal{L}(F)$, and if there is at least one pair of vector $(X_i,X_k)$ of the base such that $Hess_{g_B}f(X_i,X_k)\neq 0$, than the function $h$ only depends on the base (see Theorem \ref{teor1}). Consequently, in quasi-Einstein warped products satisfying this condition on the torsion function, the fiber is necessarily a Einstein manifold (see Corollary \ref{cor1}). Next, we consider quasi-Einstein warped products with the base conformal to a $n$-dimensional pseudo-Euclidean space and Einstein fibers, being the functions involved invariants under the action of an $(n−1)$-dimensional translation group. More precisely, let $(\mathbb{R}^n,g)$ be the pseudo-Euclidean space, $n\geq 3$, with coordinates $(x_1,\ldots,x_n)$ and $g_{ij}=\delta_{ij}\varepsilon_i$, and let $M=(\mathbb{R}^n,\overline{g})\times_fF^m$ be a warped product, where $\displaystyle \overline{g}=\frac{1}{\varphi^2}g$, $F$ is a Einstein semi-Riemannian manifold with constant Ricci curvature $\lambda_F$, $m\geq 1$, $f,\varphi,h: \mathbb{R}^n \rightarrow R$ smooth functions, and $f$ positive. In the Theorem \ref{teor2} we find necessary and sufficient conditions for the warped product metric $\widetilde{g}=\overline{g}\oplus f^2g_F$ to be quasi-Einstein; this is, so that it $\widetilde{g}$ satisfies the equation ($\ref{eqquasiEinstein}$). In the Theorem \ref{teor3} we consider $\xi = \sum\limits_{i=1}^{n}\alpha_ix_i$, $\alpha_i\in\mathbb{R}$ as a basic invariant for
an $(n-1)$-dimensional translation group, and we try to get smooth functions $f(\xi)$, $\varphi(\xi)$ and $h(\xi)$ such that $\widetilde{g}$ is quasi-Einstein with $\rho = 0$ and $F$ Ricci-flat. We first obtain necessary and sufficient conditions on $f(\xi)$, $\varphi(\xi)$ and $h(\xi)$ for the existence of $\widetilde{g}$. We show
that these conditions are different depending on the direction $\displaystyle \alpha = \sum\limits_{i=1}^{n}\alpha_i\frac{\partial}{\partial x_i}$ being null or not. In the case where $\alpha$ is not null, some solutions are explicitly given in the Theorems \ref{teor4} and \ref{teor5}, and the other are implicitly given in Theorem \ref{teortudoimplicito}. In this way, we present all solutions that are invariant under an $(n-1)$-dimensional translation group (see Theorem \ref{teorseesomentese}). Already in the case where $\alpha$ is null there are infinitely many solutions (see Theorem \ref{teor6}). We illustrate this fact with some explicit examples. Finally, in the Corollaries \ref{cor2} and \ref{cor3} we present a family of Einstein semi-Riemannian warped products whose base is not locally conformally flat. These Einstein manifolds can also be seen as multiply warped products with two fibers, and are solutions in the vacuum case ($T=0$) of the following equation
$$Ric_{\widetilde{g}} − \frac{1}{2}K\widetilde{g} = T,$$
where $K$ is the scalar curvature of $\widetilde{g}$ and $T$ is a symmetric tensor of order $2$.

\section{Main Statements}

In what follows, we state our main results. We denote by $f,_{x_ix_j}$, $\varphi,_{x_ix_j}$ and $h,_{x_ix_j}$ the second order derivative of $f$, $\varphi$ and $h$ respectively, with respect to $x_i$ and $x_j$.

\begin{theorem} \label{teor1}
Let $M=B^n\times_fF^m$ be a non trivial semi-Riemannian warped product with metric $\widetilde{g}=g_B\oplus f^2g_F$. If $h:M\rightarrow \mathbb{R}$ is a function satisfying
\begin{equation} \label{condHessh}
Hess_{\widetilde{g}}h(X,Y)=0, \ \ \ \ \forall X\in\mathcal{L}(B),Y\in\mathcal{L}(F),
\end{equation}
and if there is at least one pair of vector $(X_i,X_k)$ of the base such that $Hess_{g_B}f(X_i,X_k)\neq 0$, than $h$ depends only on the base.
\end{theorem}

As an application of Theorem \ref{teor1} we prove that under some condition on the torsion function, in a non trivial warped product quasi-Einstein, the fiber is Einstein.

\begin{corollary} \label{cor1}
Let $M=B^n\times_fF^m$ be a non trivial semi-Riemannian warped product with metric $\widetilde{g}=g_B\oplus f^2g_F$ and torsion function satisfying $Hess_{g_B}f(X_i,X_k)\neq 0$ for some pair of vectors of the base $(X_i,X_k)$. Consider $h:M\rightarrow \mathbb{R}$ a function such that $(M,\widetilde{g},h,\rho)$ is a quasi-Einstein manifold, with $\rho\in\mathbb{R}$. Then $h$ depends only on the base and consequently the fiber $F$ is an Einstein manifold.
\end{corollary}

The results obtained in Theorem \ref{teor1} and Corollary \ref{cor1} are valid in the Riemannian case. Motivated by the previous results we study the case of warped products quasi-Einstein with $h$ depending only on the base and the fiber an Einstein manifold. In addition, in accordance with the objectives of our work, we will consider $r\in\mathbb{N}$.

\begin{theorem} \label{teor2}
Let $(\R^n,g)$ be a pseudo-Euclidean space, $n\geq 3$ with coordinates $x=(x_1,\cdots, x_n)$ and $g_{ij}=\delta_{ij}\veps_i$. Consider a warped product $M = (\mathbb{R}^{n}, \overline{g})\times _{f}F^{m}$ with metric $\widetilde{g} = \overline{g} + f^{2}g_{F}$, where $\displaystyle \overline{g} =
\frac{1}{\varphi^{2}}g$,  $F$ is a Einstein semi-Riemannian manifold with constant Ricci curvature $\lambda_{F}$, $m\geq 1$, $f,\varphi, h:\mathbb{R}^{n}\rightarrow \mathbb{R}$, are smooth functions and $f$ is positive. Then  $M$ is a quasi-Einstein manifold with
$$ Ric_{\widetilde{g}} - \frac{r}{h}Hess_{\widetilde{g}}h = \rho \widetilde{g}, \ \ \ \ \ \rho\in\mathbb{R},$$
if and only if the functions $f$, $\varphi$ and $h$ satisfy:

\begin{equation} \label{eqphij}
\begin{array}{l}
(n-2)fh\varphi,_{x_ix_j}-rf\varphi h,_{x_ix_j}-mh\varphi f,_{x_ix_j}-mh\varphi,_{x_i}f,_{x_j}-mh\varphi,_{x_j}f,_{x_i}-rf\varphi,_{x_i}h,_{x_j}\\
-rf\varphi,_{x_j}h,_{x_i}=0, \ \ \ \forall i,j=1,\ldots,n, \ i\neq j,
\end{array}
\end{equation}
\begin{equation} \label{eqphii}
\begin{array}{l}
\varphi[(n-2)fh\varphi,_{x_ix_i}-rf\varphi h,_{x_ix_i}-mh\varphi f,_{x_ix_i}-2mh\varphi,_{x_i}f,_{x_i}-2rf\varphi,_{x_i}h,_{x_i}]\\
\displaystyle +\veps_i\sum\limits_{k=1}^{n}\veps_k[fh\varphi\varphi,_{x_kx_k}-(n-1)fh(\varphi,_{x_k})^2+mh\varphi\varphi,_{x_k}f,_{x_k}+rf\varphi\varphi,_{x_k}h,_{x_k}]=\veps_i\rho fh,\\
\forall i=1,\ldots,n,
\end{array}
\end{equation}
\begin{equation} \label{eqphll}
\begin{array}{l}
\displaystyle \sum\limits_{k=1}^{n}\veps_k[-fh\varphi^2f,_{x_kx_k}+(n-2)fh\varphi\varphi,_{x_k}f,_{x_k}-(m-1)h\varphi^2(f,_{x_k})^2-rf\varphi^2f,_{x_k}h,_{x_k}]\\
=h[\rho f^2-\lambda_F].
\end{array}
\end{equation}
\end{theorem}

As the problem posed is difficult to be studied in the general case, one way of treating the problem is to try to find solutions that are invariant by subgroups of isometries of the space. In this sense, we will try to find solutions invariant by the action of an $(n-1)$-dimensional translation group. 

We want to find solutions of the system (\ref{eqphij}), (\ref{eqphii}) and (\ref{eqphll}) of the form $f(\xi)$, $h(\xi)$ and $\varphi(\xi)$, where $\xi=\sum\limits_{i=1}^{n}\alpha_ix_i$, $\alpha_i\in\mathbb{R}$. Whenever $\sum\limits_{i=1}^{n}\varepsilon_i\alpha_i^2\neq 0$, without loss of generality, we may assume that $\sum\limits_{i=1}^{n}\varepsilon_i\alpha_i^2=\pm 1$. The following theorem provides the system of ordinary differential equations that must be satisfied by such solutions.

\begin{theorem} \label{teor3}
Let $(\R^n,g)$ be a pseudo-Euclidean space, $n\geq 3$ with coordinates $x=(x_1,\cdots, x_n)$ and $g_{ij}=\delta_{ij}\veps_i$. Consider a warped product $M = (\mathbb{R}^{n}, \overline{g})\times _{f}F^{m}$ with metric $\widetilde{g} = \overline{g} + f^{2}g_{F}$, where $\displaystyle \overline{g} =
\frac{1}{\varphi^{2}}g$,  $F$ is a Einstein semi-Riemannian manifold with constant Ricci curvature $\lambda_{F}$, $m\geq 1$, and $f(\xi)$, $\varphi(\xi)$ e $h(\xi)$ are smooth functions, with $\xi=\sum\limits_{i=1}^{n}\alpha_ix_i$, $\alpha_i\in\mathbb{R}$, and  $\sum\limits_{i=1}^{n}\varepsilon_i\alpha_i^2=\varepsilon_{i_0}$ or $\sum\limits_{i=1}^{n}\varepsilon_i\alpha_i^2=0$. Then the $(M,\widetilde{g},h)$ is a quasi-Einstein semi-Riemannian manifold if, only if, the functions $f$, $\varphi$ and $h$ satisty
\begin{itemize}
\item[(i)]
\begin{equation} \label{teor3sistemavetorNAOnulo}
\begin{cases}
(n-2)fh\varphi''-rf\varphi h''-mh\varphi f''-2mh\varphi'f'-2rf\varphi'h'=0\\
\displaystyle \sum\limits_{k=1}^{n}\varepsilon_k\alpha_k^2[fh\varphi\varphi''-(n-1)fh(\varphi')^2+mh\varphi\varphi'f'+rf\varphi\varphi'h']=\rho fh\\
\displaystyle \sum\limits_{k=1}^{n}\varepsilon_k\alpha_k^2[-fh\varphi^2f''+(n-2)fh\varphi\varphi'f'-(m-1)h\varphi^2(f')^2-rf\varphi^2f'h']=h[\rho f^2-\lambda_F]\\
\end{cases}
\end{equation}
whenever $\sum\limits_{i=1}^{n}\varepsilon_i\alpha_i^2=\varepsilon_{i_0}$, and
\item[(ii)]
\begin{equation} \label{teor3sistemavetornulo}
\begin{cases}
(n-2)fh\varphi'' - rf\varphi h'' - mh\varphi f'' -2mh\varphi'f'-2rf\varphi'h'=0\\
\rho=\lambda_F=0\\
\end{cases}
\end{equation}
whenever $\sum\limits_{i=1}^{n}\varepsilon_i\alpha_i^2=0$.
\end{itemize}
\end{theorem}

In the following three results we describe all the solutions of (\ref{teor3sistemavetorNAOnulo}) when $\rho=0$ and $F^m$ is a Ricci-flat manifold. In the first two theorems, we studied separately the cases $r\neq 1$ and $r=1$.

\begin{theorem} \label{teor4}
Let $(\R^n,g)$ be a pseudo-Euclidean space, $n\geq 3$ with coordinates $x=(x_1,\cdots, x_n)$ and $g_{ij}=\delta_{ij}\veps_i$. Consider smooth functions $f(\xi)$, $h(\xi)$ and $\varphi(\xi)$, where $\xi=\sum\limits_{i=1}^{n}\alpha_ix_i$, $\alpha_i\in\mathbb{R}$, and  $\sum\limits_{i=1}^{n}\varepsilon_i\alpha_i^2=\varepsilon_{i_0}$, given by
\begin{equation} \label{teor4caradasfuncoes}
\begin{cases}
f_\pm (\xi) = c_1[(a-rN_\pm)\xi+c]^{-\frac{1}{a-rN_\pm}}\\
h_\pm (\xi) = c_2[(a-rN_\pm)\xi+c]^{-\frac{N_\pm}{a-rN_\pm}}\\
\varphi_\pm (\xi) = c_3[(a-rN_\pm)\xi+c]^{-\frac{k}{a-rN_\pm}}
\end{cases},
\end{equation}
with $k,c_1,c_2,c_3,N_\pm \in\mathbb{R}$, $k,c_1,c_2,c_3>0$, 
$$N_\pm=\frac{r(k+a)\pm\sqrt{r^2(k+a)^2-r(r-1)(a^2-b)}}{r(r-1)}, \ \ \ r>0, \ \ r\neq 1, \ \ \ r(k+a)^2\geq (r-1)(a^2-b),$$
$a=(n-2)k-m$ and $b=m(2k+1)-(n-2)k^2$. Then $(\mathbb{R}^{n}\times _{f}F^{m},\widetilde{g},h)$ is a quasi-Einstein semi-Riemannian manifold, with $\widetilde{g} = \overline{g} + f^{2}g_{F}$, $\displaystyle \overline{g} = \frac{1}{\varphi^{2}}g$ and $F$ Ricci-flat. These solutions are defined on the half space defined by $\displaystyle \sum\limits_{i=1}^{n}\alpha_ix_i>-\frac{c}{a-rN_\pm}$.
\end{theorem}

\begin{theorem} \label{teor5}
Let $(\R^n,g)$ be a pseudo-Euclidean space, $n\geq 3$ with coordinates $x=(x_1,\cdots, x_n)$ and $g_{ij}=\delta_{ij}\veps_i$. Consider smooth functions $f(\xi)$, $h(\xi)$ and $\varphi(\xi)$, where $\xi=\sum\limits_{i=1}^{n}\alpha_ix_i$, $\alpha_i\in\mathbb{R}$, and  $\sum\limits_{i=1}^{n}\varepsilon_i\alpha_i^2=\varepsilon_{i_0}$, given by
\begin{equation} \label{teor5caradasfuncoes}
\begin{cases}
f(\xi) = c_1[(a-N)\xi+c]^{-\frac{1}{a-N}}\\
h(\xi) = c_2[(a-N)\xi+c]^{-\frac{N}{a-N}}\\
\varphi(\xi) = c_3[(a-N)\xi+c]^{-\frac{k}{a-N}}
\end{cases},
\end{equation}
with $k,c_1,c_2,c_3,N_\pm \in\mathbb{R}$, $k,c_1,c_2,c_3>0$, 
$$N=\frac{a^2-b}{2(k+a)}, \ \ \ k+a\neq 0, \ \ \ r=1,$$
$a=(n-2)k-m$ and $b=m(2k+1)-(n-2)k^2$. Then $(\mathbb{R}^{n}\times _{f}F^{m},\widetilde{g},h)$ is a quasi-Einstein semi-Riemannian manifold, with $\widetilde{g} = \overline{g} + f^{2}g_{F}$, $\displaystyle \overline{g} = \frac{1}{\varphi^{2}}g$ and $F$ Ricci-flat. These solutions are defined on the half space defined by $\displaystyle \sum\limits_{i=1}^{n}\alpha_ix_i>-\frac{2c(k+a)}{a^2+b+2ka}$.
\end{theorem}

\begin{theorem} \label{teortudoimplicito}
Let $(\R^n,g)$ be a pseudo-Euclidean space, $n\geq 3$ with coordinates $x=(x_1,\cdots, x_n)$ and $g_{ij}=\delta_{ij}\veps_i$. Consider the functions $x(\xi)$ and $z(\xi)$,where $\xi=\sum\limits_{i=1}^{n}\alpha_ix_i$, $\alpha_i\in\mathbb{R}$, and  $\sum\limits_{i=1}^{n}\varepsilon_i\alpha_i^2=\varepsilon_{i_0}$, given by
\begin{equation} \label{teortudoimplicitoeqs1}
\begin{cases}
v(z)z'- c(a-rz)e^{\int v(z)dz}=0, \ \ \ \ \ c>0,\\
x=ce^{\int v(z)dz}
\end{cases}
\end{equation}
where $\displaystyle v(z)=\frac{r(a-rz)}{r(r-1)z^2-2r(k+a)z+a^2-b}$, with $k>0$, $a=(n-2)k-m$, $b=m(2k+1)-(n-2)k^2$, $n,m\in\mathbb{N}$, $n\geq 3$. Let $f(\xi)$, $\varphi(\xi)$ and $h(\xi)$ be functions obtained by integrating
\begin{equation} \label{teortudoimplicitofhvarphi}
\frac{f'}{f}(\xi)=x(\xi), \ \ \ \ \ \frac{\varphi'}{\varphi}(\xi)=kx(\xi), \ \ \ \ \ \frac{h'}{h}(\xi)=x(\xi)z(\xi).
\end{equation}
Then $(\mathbb{R}^{n}\times _{f}F^{m},\widetilde{g},h)$ is a quasi-Einstein semi-Riemannian manifold, with $\widetilde{g} = \overline{g} + f^{2}g_{F}$, $\displaystyle \overline{g} = \frac{1}{\varphi^{2}}g$ and $F$ Ricci-flat.
\end{theorem}

\begin{theorem} \label{teorseesomentese}
Let $(\R^n,g)$ be a pseudo-Euclidean space, $n\geq 3$ with coordinates $x=(x_1,\cdots, x_n)$ and $g_{ij}=\delta_{ij}\veps_i$. Let $(\mathbb{R}^{n}\times _{f}F^{m},\widetilde{g},h)$ be a quasi-Einstein semi-Riemannian manifold, with $\widetilde{g} = \overline{g} + f^{2}g_{F}$, $\displaystyle \overline{g} = \frac{1}{\varphi^{2}}g$ and $F$ Ricci-flat. Then $f$, $\varphi$ and $h$ are invariant under an $(n−1)$-dimensional translation group whose basic invariant is $\xi=\sum\limits_{i=1}^{n}\alpha_ix_i$, $\alpha_i\in\mathbb{R}$, where $\alpha=\sum_{i=1}^{n}\alpha_i\frac{\partial}{\partial x_i}$ is a non-null vector if, and only if, $f$, $\varphi$ and $h$ are given as in Theorems \ref{teor4}, \ref{teor5} or \ref{teortudoimplicito}.
\end{theorem}

The following theorem shows that there are infinitely many quasi-Einstein semi-Riemannian warped product $(M=\mathbb{R}^n\times_fF^m,\widetilde{g},h)$ with zero $r$-Bakry-Emery tensor, where $\widetilde{g} = \overline{g} + f^{2}g_{F}$, $\overline{g}=\frac{1}{\varphi^2}g$ and $F$ is a Ricci-flat manifold, which are invariant under the action of an $(n−1)$–dimensional group acting on $\mathbb{R}^n$, when $\displaystyle \alpha=\sum_{i=1}^{n}\alpha_i\frac{\partial}{\partial x_i}$ is a null vector.

\begin{theorem} \label{teor6}
Let $(\R^n,g)$ be a pseudo-Euclidean space, $n\geq 3$ with coordinates $x=(x_1,\cdots, x_n)$ and $g_{ij}=\delta_{ij}\veps_i$. Consider $f(\xi)$ and $\varphi(\xi)$ any positive differentiable functions, where $\xi=\sum\limits_{i=1}^{n}\alpha_ix_i$ and $\sum\limits_{i=1}^{n}\varepsilon_i\alpha_i^2=0$, and $h(\xi)$ given by ordinary differential equation
\begin{equation} \label{teor6equacao}
-rf\varphi h''-2rf\varphi'h'+[(n-2)f\varphi''-m\varphi f''-2m\varphi'f']h=0.
\end{equation}
Then $(M=\mathbb{R}^n\times_fF^m,\widetilde{g},h)$ is a quasi-Einstein semi-Riemannian manifold, where $\widetilde{g} = \overline{g} + f^{2}g_{F}$, $\overline{g}=\frac{1}{\varphi^2}g$ and $F$ is a Ricci-flat manifold.
\end{theorem}

Let us present two examples illustrating the Theorem \ref{teor6}. Let $f(\xi)=k_1e^{A\xi}$ and $\varphi(\xi)=k_2e^{B\xi}$, where $k_1,k_2>0$, $A,B\in\mathbb{R}$ with $C=r^2B^2+r[(n-2)B^2-mA^2-2mAB]\geq 0$, $\xi=\sum\limits_{i=1}^{n}\alpha_ix_i$ and $\sum\limits_{i=1}^{n}\varepsilon_i\alpha_i^2=0$. In this case, the equation (\ref{teor6equacao}) becomes
$$k_1k_2e^{(A+B)\xi}\{-rh''-2rBh'+[(n-2)B^2-mA^2-2mAB]h\}=0.$$
Then $h$ is given by $h(\xi)=c_1e^{[\frac{-rB+\sqrt{C}}{r}]\xi}+c_2e^{[\frac{-rB-\sqrt{C}}{r}]\xi}$, where $c_1,c_2\in\mathbb{R}$. By Theorem \ref{teor6}, $(M=\mathbb{R}^n\times_fF^m,\widetilde{g},h)$ is a quasi-Einstein manifold with $\rho=0$ and $F$ Ricci-flat.

In \cite{romildobenedito}, the authors considered $\varepsilon_1 = -1$, $\varepsilon_2 = 1$, $\alpha_1 = \alpha_2 = 1$, and $\alpha_l = 0$ with $3 \leq l \leq n$. Then $\displaystyle \sum_{k=1}^{n}\varepsilon_k\alpha_k^2=0$ and $\xi=x_1+x_2$, and under these conditions they showed that if $f(\xi)=e^{(\sqrt{n-1}-1)\xi}$ and $\varphi(\xi)=e^{\xi}$, the warped product $(\mathbb{R}^n,\overline{g})\times_f\mathbb{R}$ with metric $\widetilde{g}=\overline{g}\oplus f^2(-dt^2)$ is complete.
We observed that this is a particular case from the previous example. In fact, just do $A=\sqrt{n-1}-1$ e $B=k_1=k_2=m=1$. In this case, we obtain that $C=r(r-1)$ and $h(\xi)=c_1e^{[\frac{-r+\sqrt{r(r-1)}}{r}]\xi}+c_2e^{[\frac{-r-\sqrt{r(r-1)}}{r}]\xi}$, with $c_1,c_2\in\mathbb{R}$. Therefore,
$$(M,\widetilde{g},h) = \left(\mathbb{R}^n\times_f\mathbb{R},\frac{1}{\varphi^2}\oplus f^2(-dt^2),c_1e^{[\frac{-r+\sqrt{r(r-1)}}{r}]\xi}+c_2e^{[\frac{-r-\sqrt{r(r-1)}}{r}]\xi}\right)$$
is a complete quasi-Einstein manifold, where $f(\xi)=e^{(\sqrt{n-1}-1)\xi}$, $\varphi(\xi)=e^{\xi}$, $\displaystyle \sum_{k=1}^{n}\varepsilon_k\alpha_k^2=0$ and $\xi=x_1+x_2$.

Now, if we choose $f(\xi)=\varphi(\xi)=\xi^2$, where $\xi\in\mathbb{R}^n\setminus\{0\}$, $\xi=\sum\limits_{i=1}^{n}\alpha_ix_i$ and $\sum\limits_{i=1}^{n}\varepsilon_i\alpha_i^2=0$, (\ref{teor6equacao}) becomes
$$\xi^2h''+4\xi h'+\frac{10m-2(n-2)}{r}h=0,$$
which is an Cauchy-Euler equation. Then $h$ is given by
$$h=\begin{cases}
|\xi|^{-\frac{3}{2}}(c_1|\xi|^\lambda+c_2|\xi|^{-\lambda}), & if \ \ 9>\frac{40m-8(n-2)}{r}\\
|\xi|^{-\frac{3}{2}}(c_1+c_2\ln|\xi|), & if \ \ 9=\frac{40m-8(n-2)}{r}\\
|\xi|^{-\frac{3}{2}}\left[c_1\sin(\lambda\ln|\xi|)+c_1\cos(\lambda\ln|\xi|)\right], & if \ \ 9<\frac{40m-8(n-2)}{r}
\end{cases},$$
where $\displaystyle \lambda = \frac{1}{2}\left|9-\frac{40m-8(n-2)}{r}\right|^{\frac{1}{2}}$ and $c_1,c_2\in\mathbb{R}$. By Theorem \ref{teor6}, $(M=\mathbb{R}^n\setminus\{0\}\times_fF^m,\widetilde{g},h)$ is a quasi-Einstein manifold with $\rho=0$ and $F$ Ricci-flat.

In the next two results we used the quasi-Einstein manifolds explicitly obtained in Theorems \ref{teor4} and \ref{teor5}, and implicitly obtained in Theorems \ref{teortudoimplicito} and \ref{teor6}, to build Ricci-flat warped product Einstein manifolds whose base is not locally conformally flat.

\begin{corollary} \label{cor2}
Let $(\R^n,g)$ be a pseudo-Euclidean space, $n\geq 3$ with coordinates $x=(x_1,\cdots, x_n)$ and $g_{ij}=\delta_{ij}\veps_i$. Consider smooth functions $f(\xi)$, $h(\xi)$ and $\varphi(\xi)$, where $\xi=\sum\limits_{i=1}^{n}\alpha_ix_i$, $\alpha_i\in\mathbb{R}$, and  $\sum\limits_{i=1}^{n}\varepsilon_i\alpha_i^2=\varepsilon_{i_0}$, explicitly given by (\ref{teor4caradasfuncoes}) or (\ref{teor5caradasfuncoes}), or implicitly given by (\ref{teor6equacao}) and defined in space $\{\xi\in\mathbb{R}^n/f(\xi)>0, h(\xi)>0,\varphi(\xi)\neq 0\}$. Then $(\mathbb{R}^n\times_fF_1^m)\times_hF_2^r$ with metric $\widetilde{g}=(\overline{g}\oplus f^2g_{F_1})\oplus h^2g_{F_2}$ is a Ricci-flat Einstein semi-Riemannian manifold, where $F_1^m$ and $F_2^r$ are Ricci-flat Einstein manifold of dimension $m$ and $r$ respectively.
\end{corollary}

The solutions of the ordinary differential equation (\ref{teor6equacao}) generate infinite Einstein manifolds as given in Corollary \ref{cor2}. To exemplify this result, we recall that if $f(\xi)=k_1e^{A\xi}$ and $\varphi(\xi)=k_2e^{B\xi}$, where $k_1,k_2>0$, $A,B\in\mathbb{R}$ with $(n-2)B^2-mA^2-2mAB\geq 0$, $\xi=\sum\limits_{i=1}^{n}\alpha_ix_i$ and $\sum\limits_{i=1}^{n}\varepsilon_i\alpha_i^2=0$, then $h$ is given by $h(\xi)=c_1e^{[\frac{-rB+\sqrt{C}}{r}]\xi}+c_2e^{[\frac{-rB-\sqrt{C}}{r}]\xi}$, where $C=r^2B^2+r[(n-2)B^2-mA^2-2mAB]$, $c_1,c_2\in\mathbb{R}$. Thus $(\mathbb{R}^n\times_fF_1^m,\overline{g}\oplus g_{F_1},h)$ is a quasi-Einstein manifold with $\rho=0$ and $F_1$ Ricci-flat (see example of the Theorem \ref{teor6}). Now, if we choose a Ricci-flat manifold $(F_2^r,g_{F_2})$, by the Corollary \ref{cor2} we have that $(\mathbb{R}^n\times_fF_1^m)\times_hF_2^r$ is a Ricci-flat Einstein manifold.

\begin{corollary} \label{cor3}
Let $(\R^n,g)$ be a pseudo-Euclidean space, $n\geq 3$ with coordinates $x=(x_1,\cdots, x_n)$ and $g_{ij}=\delta_{ij}\veps_i$. Consider smooth functions $f(\xi)$, $h(\xi)$ and $\varphi(\xi)$, where $\xi=\sum\limits_{i=1}^{n}\alpha_ix_i$, $\alpha_i\in\mathbb{R}$, implicitly given by (\ref{teortudoimplicitofhvarphi}), defined in space $\{\xi\in\mathbb{R}^n/f(\xi)>0, h(\xi)>0,\varphi(\xi)\neq 0\}$. Then $(\mathbb{R}^n\times_fF_1^m)\times_hF_2^r$ with metric $\widetilde{g}=(\overline{g}\oplus f^2g_{F_1})\oplus h^2g_{F_2}$ is a Ricci-flat Einstein semi-Riemannian manifold, where $F_1^m$ is a Ricci-flat Einstein manifold of dimension $m$ and $F_2^r$ is a Einstein manifold of dimension $r$.
\end{corollary}

\begin{remark}
The functions $h_\pm$ and $h$, obtained in the Theorems \ref{teor4} and \ref{teor5} respectively, are positives in their definition spaces. For this reason we could use them as torsion functions in Corollary \ref{cor2}.
\end{remark}

\begin{remark}
In the warped product Einstein manifolds obtained in the Corollaries \ref{cor2} and \ref{cor3} the bases $B=(\mathbb{R}^n \times_fF_1^m)$ are not locally conformally flat, due to the result in \cite{BV-GR-VL}.
\end{remark}

\begin{remark}
Note that, as the functions $f$ and $h$ in the Corollaries \ref{cor2} and \ref{cor3} are defined only in space $\mathbb{R}^n$, we have to the Einstein manifolds  $(\mathbb{R}^n\times_fF_1^m)\times_hF_2^1$ are multiply warped products, where the base is $(\mathbb{R}^n,\overline{g})$ and the fibers are $F_1$ and $F_2$ with torsion functions $f$ and $h$, respectively. This is,
$$((\mathbb{R}^n\times_fF_1^m)\times_hF_2^1,(\overline{g}\oplus f^2g_{F_1})\oplus h^2g_{F_2})=(\mathbb{R}^n\times_fF_1^m\times_hF_2^1,\overline{g}\oplus f^2g_{F_1}\oplus h^2g_{F_2}).$$
\end{remark}

 \section{Proofs of the Main Results}
 
\vspace{.2in}

\begin{proof}[Proof of Theorem \ref{teor1}:] 

Let $M=B^n\times_fF^m$ a warped product and $h:M\rightarrow\mathbb{R}$ satisfying
$$Hess_{\widetilde{g}}h(X,Y)=0, \ \ \ \ \forall X\in\mathcal{L}(B),Y\in\mathcal{L}(F).$$
Considering $X_1,\ldots,X_n\in\mathcal{L}(B)$ and $Y_1,\ldots,Y_m\in\mathcal{L}(F)$, where $\mathcal{L}(B)$ and $\mathcal{L}(F)$ are respectively the lift of a vector field on $B$ and $F$ to $B\times F$, we have

\begin{equation} \label{Ricprodtorc}
\begin{cases}
Ric_{\widetilde{g}}(X_i,X_j) = Ric_{g_B}(X_i,X_j) - \frac{m}{f}Hess_{g_B}f(X_i,X_j), \ \forall i,j=1,\ldots,n\\
Ric_{\widetilde{g}}(X_i,Y_j)=0, \ \forall i=1,\ldots,n, j=1,\ldots,m\\
Ric_{\widetilde{g}}(Y_i,Y_j)=Ric_{g_F}(Y_i,Y_j) - \left[\frac{\Delta_{g_B}f}{f}+(m-1)\widetilde{g}\frac{(grad_{\widetilde{g}}f,grad_{\widetilde{g}}f)}{f^2}\right]\widetilde{g}(Y_i,Y_j), \ \forall i,j=1,\ldots,m
\end{cases}
\end{equation}

How for hypothesis
\begin{equation} \label{teor2eq13}
Hess_{\widetilde{g}}h(X_i,Y_j)=0,
\end{equation}

and for definition

$$Hess_{\widetilde{g}}h(X_i,Y_j)=X_iY_j(h)-(\nabla_{X_i}Y_j)(h),$$

$\forall i=1,\ldots,n, j=1,\ldots,m$, where $\nabla$ is the connection of M, since (see \cite{O'neil})
$$\nabla_{X_i}Y_j=\frac{X_i(f)}{f}Y_j,$$
we have
\begin{equation} \label{teor2eq14}
Hess_{\widetilde{g}}h(X_i,Y_j)=h,_{x_iy_j}-\frac{f,_{x_i}}{f}h,_{y_j}.
\end{equation}

Using (\ref{teor2eq13}) and (\ref{teor2eq14}) we obtain

\begin{equation} \label{teor2eq15}
h,_{x_iy_j}-\frac{f,_{x_i}}{f}h,_{y_j}=0.
\end{equation}

If $h=h(x_1,\ldots,x_n)$ then the equation (\ref{teor2eq15}) is trivially satisfied. Suppose that there is at least one $y_j$, with $1\leq j \leq m$ such that $h,_{y_j} \neq 0$ in a point $p\in M$. Then $h,_{y_j} \neq 0$ in a neighbourhood $V_p$ of $p$. In this case, it follows from (\ref{teor2eq15}) that

\begin{equation} \label{teor2eq16}
\frac{h,_{x_iy_j}}{h,_{y_j}}=\frac{f,_{x_i}}{f}, \ \ \ \forall i=1,\ldots,n, j=1,\ldots,m.
\end{equation}

Integrating (\ref{teor2eq16}) in relation to $x_i$ we obtain
$$\ln h,_{y_j} = \ln f + l(\hat{x_i}),$$
this is,
\begin{equation} \label{teor2eq17}
h,_{y_j}=fe^{l(\hat{x_i})}.
\end{equation}
Fixing $i$ and $j$ in (\ref{teor2eq17}) and deriving in relation to $x_k$ with $k\neq i$, we obtain
$$h,_{y_jx_k}=f,_{x_k}e^{l(\hat{x_i})} + fl,_{x_k}e^{l(\hat{x_i})},$$
which is equivalent to
$$\frac{h,_{x_ky_j}}{h,_{y_j}}=\frac{f,_{x_k}}{f} + l,_{x_k}.$$
Using (\ref{teor2eq16}) we have
$$l,_{x_k}=0,$$
and this means that $l$ does not depend $x_k$, this is
$$l=l(\hat{x_i},\hat{x_k}).$$

Repeating this process we obtain that $l$ depends only on the fiber. Therefore
\begin{equation} \label{teor2eq18}
h,_{y_j}=fe^{l(y_1,\ldots,y_m)},
\end{equation}
with $1\leq j \leq m$. Integrating (\ref{teor2eq18}) in relation to $y_j$, we obtain
\begin{equation} \label{teor2eq19}
h(x_1,\ldots,x_n,y_1,\ldots,y_m) = f(x_1,\ldots,x_n)\int e^{l(y_1,\ldots,y_m)}dy_j + m(\hat{y_j}).
\end{equation}

Using the first equation of (\ref{Ricprodtorc}) we have
$$Hess_{\widetilde{g}}h(X_i,X_k) = \rho g_B(X_i,X_k) - Ric_{g_B}(X_i,X_k) + \frac{m}{f}Hess_{g_B}f(X_i,X_k), \ \forall i,k=1,\ldots n,$$
proving that $Hess_{\widetilde{g}}h(X_i,X_k)$ depends only on the base, $\forall i,k=1,\ldots n$. Thus, considering $j$ fixed in equation (\ref{teor2eq19}) we have
\begin{equation} \label{teor2eq20}
\frac{\partial}{\partial y_j}Hess_{\widetilde{g}}h(X_i,X_k) = 0, \ \ \ \forall i,k=1,\ldots n.
\end{equation}

On the other hand, $\forall i,k=1,\ldots n$ we have
$$\frac{\partial}{\partial y_j}Hess_{\widetilde{g}}h(X_i,X_k) = \frac{\partial}{\partial y_j}\left[Hess_{\widetilde{g}}f(X_i,X_k)\int e^{l(y_1,\ldots,y_m)}dy_j + Hess_{\widetilde{g}}m(X_i,X_k)\right].$$
Since $m=m(\hat{y_j})$, using the definition of the Hessian get that
\begin{equation*}
\frac{\partial}{\partial y_j}Hess_{\widetilde{g}}m(X_i,X_k)=0, \ \ \ \forall i,k=1,\ldots,n,
\end{equation*}
and as $Hess_{\widetilde{g}}f(X_i,X_k) = Hess_{g_B}f(X_i,X_k)$ follow that
\begin{equation} \label{teor2eq22}
\frac{\partial}{\partial y_j}Hess_{\widetilde{g}}h(X_i,X_k) = Hess_{g_B}f(X_i,X_k)e^{l(y_1,\ldots,y_m)}, \ \ \ \forall i,k=1,\ldots,n.
\end{equation}

Using (\ref{teor2eq20}) and (\ref{teor2eq22}) we obtain that
$$Hess_{g_B}f(X_i,X_k)e^{l(y_1,\ldots,y_m)}=0, \ \ \ \forall i,k=1,\ldots,n.$$
We have for hypothesis that there is at least one pair of vector $(X_i,X_k)$ of the base such that $Hess_{g_B}f(X_i,X_k) \neq 0$. Then
$$e^{l(y_1,\ldots,y_m)}=0,$$
but this is impossible. Therefore $h,_{y_j}=0$, $\forall j=1,\ldots,m$ e $\forall p\in M$. Consequently $h$ depends only on the base. This concludes the proof of the Theorem \ref*{teor1}.
\end{proof}

\vspace{.2in}

\begin{proof}[Proof of Corollary \ref{cor1}]
If $(M,\widetilde{g},h,\rho)$ is a quasi-Einstein manifold, $M$ satisfy
\begin{equation} \label{cor1eqquasiEinstein}
Ric_{\widetilde{g}} - \frac{r}{h}Hess_{\widetilde{g}}h=\rho \widetilde{g}, \ \ \ \ \ r\in\mathbb{R}^*_+, \ \ \ \ \ \rho\in\mathbb{R}.
\end{equation}
But, for $Y,Z\in\mathcal{L}(F)$, we have
$$\widetilde{g}(Y,Z)=f^2g_F(Y,Z)$$
and
$$Ric_{\widetilde{g}}(Y,Z)=Ric_{g_F}(Y,Z)−(f grad_{g_B}f + (m − 1)|grad_{g_B}f|^2)g_F(Y,Z)$$
(see for example \cite{O'neil}). Replacing $Ric_{\widetilde{g}}(Y,Z)$ in (\ref{cor1eqquasiEinstein}) we have
\begin{equation}\label{cor1eq15'}
Ric_{g_F}(Y,Z)= (\rho f^2 + f grad_{g_B}f + (m − 1)|grad_{g_B}f|^2)g_F(Y,Z)+\frac{r}{h}Hess_{\widetilde{g}}h(Y,Z).
\end{equation}

It follows from (\ref{cor1eq15'}) that F is Einstein if and only if 
$$Hess_{\widetilde{g}}h(Y,Z)=\lambda g_F.$$
Indeed, for $X\in\mathcal{L}(B)$ e $Y\in\mathcal{L}(F)$, as $Ric_{\widetilde{g}}(X,Y)=0$ and $\widetilde{g}(X,Y)=0$, by (\ref{cor1eqquasiEinstein}) we have to $h$ satisfies the equation (\ref{condHessh}). By Theorem \ref{teor1}, follow that $h$ depends only on the base, and thus $grad_{\widetilde{g}}h=grad_{g_B}h$. This shows us that the horizontal component of the field $grad_{\widetilde{g}}h$ is $grad_{g_B}h$ and the vertical component is null. Therefore, by warped product metric and by definition of Hessian tensor,
$$\widetilde{\nabla}_Y(grad_{\widetilde{g}}h)=\frac{grad_{g_B}h(f)}{f}Y$$
and
$$Hess_{\widetilde{g}}h(Y,Z)=\frac{grad_{g_B}h(f)}{f}\widetilde{g}(Y,Z)=fgrad_{g_B}h(f)g_F(Y,Z).$$
This concludes the proof of Corollary \ref{cor1}.
\end{proof}

\vspace{.2in}

\begin{proof}[Proof of Theorem \ref{teor2}:]

Assume initially that $m > 1$. It follows from \cite{O'neil} that if $X,Y\in{\cal
L}(\mathbb{R}^{n})$ and $V,W\in{\cal
L}(F)$, where ${\cal L}(\mathbb{R}^{n})$ and ${\cal L}(F)$ are respectively  the spaces of lifts of vector fields on $\mathbb{R}^{n}$ and $F$ to $\mathbb{R}^{n}\times F$, then
\begin{eqnarray}
\label{ric1}
\left\{\begin{array}{lcl}
Ric_{\widetilde{g}}(X,Y) &=& \displaystyle Ric_{\overline{g}}(X,Y) - \frac{m}{f}Hess_{\overline{g}}f(X,Y) \\
    Ric_{\widetilde{g}}(X,V) &=& 0 \\
    Ric_{\widetilde{g}}(V,W) &=& \displaystyle Ric_{g_{F}}(V,W) -
  \widetilde{g}(V,W)\left[\frac{\triangle_{\overline{g}}f}{f} + (m-1)\frac{\widetilde{g}(grad_{\overline{g}} f, grad_{\overline{g}}
  f)}{f^{2}}\right]
\end{array}
\right.
\end{eqnarray}

 It is well known (see, e.g., \cite{Be}) that if $\overline{g}=\frac{1}{\varphi^2}g$  , then
\[
Ric_{\overline{g}}=\frac{1}{\varphi^2}\left \{(n-2)\varphi
Hess_{g}\varphi+[\varphi \Delta_g
\varphi-(n-1)|grad_g\varphi|^2]g \right \}\,.
\]
Considering a parameterization $(x_1,\ldots,x_n,y_1,\ldots,y_n)$ of $M$, and denoting by $\displaystyle X_i=\frac{\partial}{\partial x_i}$, $\displaystyle Y_j=\frac{\partial}{\partial y_j}$, $i=1,\ldots,n$, $j=1,\ldots,m$, since $g(X_{i},X_{j}) = \varepsilon_{i}\delta_{ij}$, we have

\begin{eqnarray*}
  Ric_{\overline{g}}(X_{i}, X_{j}) &=& \frac{1}{\varphi}\left \{(n-2)
Hess_{g}\varphi(X_{i}, X_{j}) \right \}\ \forall\ i\neq j = 1,\ldots n,  \\
  Ric_{\overline{g}}(X_{i}, X_{i}) &=& \frac{1}{\varphi^2}\left \{(n-2)\varphi
Hess_{g}\varphi(X_{i}, X_{i})+[\varphi \Delta_g
\varphi-(n-1)|grad_g\varphi|^2]\varepsilon_{i} \right \}\ \forall\
i=1,\ldots n.
\end{eqnarray*}
 Since $Hess_{g}\varphi(X_{i}, X_{j}) = \varphi_{,x_{i}x_{j}}$
, $\Delta_g \varphi = \displaystyle\sum_{k=1}^{n}\varepsilon_{k}\varphi_{,x_{k}x_{k}}$ and
$|grad_g\varphi|^2 = \displaystyle\sum_{k =
1}^{n}\varepsilon_{k}\varphi_{,x_{k}}^{2} $, we have
\begin{equation}\label{ric2}
\left\{ \begin{array}{ccl}
  Ric_{\overline{g}}(X_{i}, X_{j}) &=& \displaystyle\frac{(n-2)\varphi_{,x_{i}x_{j}}}{\varphi}\qquad \forall\ i\neq j=1\ldots\ n, \\
  Ric_{\overline{g}}(X_{i}, X_{i}) &=& \displaystyle\frac{(n-2)\varphi_{,x_{i}x_{i}} +
  \varepsilon_{i}\displaystyle\sum_{k=1}^{n}\varepsilon_{k}\varphi_{,x_{k}x_{k}}}{\varphi}
  - (n - 1)\varepsilon_{i}\sum_{k=1}^{n}\frac{\varepsilon_{k}\varphi_{,x_{k}}^{2}}{\varphi^{2}} \ \forall i=1,\ldots,n.
\end{array}
\right.
\end{equation}

Recall that
\[
Hess_{\overline{g}}f(X_{i},X_{j})=f_{ ,x_ix_j}-\sum_k
\overline{\Gamma}_{ij}^k f_{,x_k},
\]
where $\overline{\Gamma}_{ij}^k$ are the Christoffel symbols of the metric $\overline{g}$. For $i,\ j,\ k$ distinct, we have
\[
\overline{\Gamma}_{ij}^k= 0\ \ \ \ \ \ \ \ \ \overline{\Gamma}_{ij}^i= -\frac{\varphi_{,x_{j}}}{\varphi}\ \ \ \ \ \ \ \overline{\Gamma}_{ii}^k= \varepsilon_{i}\varepsilon_{k}\frac{\varphi_{,x_{k}}}{\varphi}\ \ \ \ \ \ \ \overline{\Gamma}_{ii}^i= -\frac{\varphi_{,x_{j}}}{\varphi} .\]
Therefore,

\begin{equation}\label{hes1}
\left\{
\begin{array}{ccc}
Hess_{\overline{g}}f(X_{i},X_{j}) &=& \displaystyle
f_{,x_ix_j}+\frac{\varphi_{,x_j}}{\varphi}f_{,x_i}
+\frac{\varphi_{,x_i}}{\varphi}f_{,x_j},\ \forall\ i\neq j = 1\ldots n,\\
Hess_{\overline{g}}f(X_{i},X_{i}) &=& \displaystyle
f_{,x_ix_i}+2\frac{\varphi_{,x_i}}{\varphi}f_{,x_i} -
\varepsilon_{i}\sum_{k=1}^{n}\varepsilon_{k}\frac{\varphi_{,x_k}}{\varphi}f_{,x_k}, \ \forall\ i = 1\ldots n.
\end{array}
\right.
\end{equation}

 Substituting (\ref{ric2}) and (\ref{hes1}) into the first equation of (\ref{ric1}) we  obtain
 \begin{equation}\label{ric3}
  Ric_{\widetilde{g}}(X_{i},X_{j}) =
 \frac{(n-2)\varphi_{,x_{i}x_{j}}}{\varphi}-\frac{m}{f}\left[f_{,x_ix_j}+\frac{\varphi_{,x_j}}{\varphi}f_{,x_i}
+\frac{\varphi_{,x_i}}{\varphi}f_{,x_j}\right],\ \forall\ i\neq j
\end{equation}

and

\begin{eqnarray}\label{ric4}
  Ric_{\widetilde{g}}(X_{i},X_{i}) &=& \frac{(n-2)\varphi_{,x_{i}x_{i}} +
  \varepsilon_{i}\displaystyle\sum_{k=1}^{n}\varepsilon_{k}\varphi_{,x_{k}x_{k}}}{\varphi}
  - (n -
  1)\varepsilon_{i}\sum_{k=1}^{n}\frac{\varepsilon_{k}\varphi_{,x_{k}}^{2}}{\varphi^{2}} \nonumber\\
  &-&  \frac{m}{f}\left[f_{,x_ix_i}+2\frac{\varphi_{,x_i}}{\varphi}f_{,x_i} -
\varepsilon_{i}\sum_{k=1}^{n}\varepsilon_{k}\frac{\varphi_{,x_k}}{\varphi}f_{,x_k}\right].
\end{eqnarray}

On the other hand,
\begin{equation}\label{gr}
\left\{
\begin{array}{ccl}
  Ric_{{g_{F}}}(Y_{i}, Y_{j}) &=& \lambda_{F}g_{F}(Y_{i}, Y_{j}) \\
  \widetilde{g}(Y_{i}, Y_{j}) &=& f^{2}g_{F}(Y_{i}, Y_{j}) \\
  \Delta_{\overline{g}}f &=& \varphi^{2}\sum_{k = 1}^{n}\varepsilon_{k}f_{,x_{k}x_{k}} - (n-2)\varphi\sum_{k = 1}^{n}\varepsilon_{k}\varphi_{,x_k}f_{,x_k} \\
  \widetilde{g}(grad_{\overline{g}}f, grad_{\overline{g}} f) &=& \varphi^{2}\sum_{k =
  1}^{n}\varepsilon_{k}f_{,x_{k}}^{2}
\end{array}
\right.,
\end{equation}

$\forall i,j=1,\ldots,m$. Substituting (\ref{gr}) in the third equation of (\ref{ric1}),
we have

\begin{equation}\label{ric5}
 Ric_{\widetilde{g}}(Y_{i},Y_{j}) = \gamma_{ij}g_{F}(Y_{i},Y_{j})
\end{equation}
where \[\gamma_{ij } =\lambda_{F} -
f\varphi^{2}\sum_{k = 1}^{n}\varepsilon_{k}f_{,x_{k}x_{k}} +
(n-2)f\varphi\sum_{k = 1}^{n}\varepsilon_{k}
\varphi_{,x_k}f_{,x_k} - (m - 1)\varphi^{2}\sum_{k =
1}^{n}\varepsilon_{k}f_{,x_{k}}^{2}  .\]

On the other hand, since $(M,\widetilde{g},h)$ is quasi-Einstein, we have

\begin{equation}\label{quasiEinstein}
Ric_{\widetilde{g}}-\frac{r}{h}Hess_{\widetilde{g}}h=\rho \widetilde{g},
\end{equation}

and as $h:\mathbb{R}^{n}\rightarrow \mathbb{R}$, we have 

\[Hess_{\widetilde{g}}h(X_{i}, X_{j}) = Hess_{\overline{g}}h(X_{i}, X_{j}), \forall 1\leq i, j\leq n,\]
i.e.,
\begin{equation}\label{hess3}
\left\{
\begin{array}{ccc}
Hess_{\widetilde{g}}h(X_{i},X_{j}) &=& \displaystyle
h_{,x_ix_j}+\frac{\varphi_{,x_j}}{\varphi}h_{,x_i}
+\frac{\varphi_{,x_i}}{\varphi}h_{,x_j},\ \forall\ i\neq j = 1\ldots n,\\
Hess_{\widetilde{g}}h(X_{i},X_{i}) &=& \displaystyle
h_{,x_ix_i}+2\frac{\varphi_{,x_i}}{\varphi}h_{,x_i} -
\varepsilon_{i}\sum_{k=1}^{n}\varepsilon_{k}\frac{\varphi_{,x_k}}{\varphi}h_{,x_k}, \ \forall i=1,\ldots,n.
\end{array}
\right.
\end{equation}

By substituting (\ref{ric3}) and the first equation of (\ref{hess3}) into (\ref{quasiEinstein}), we obtain (\ref{eqphij}). Again using (\ref{ric4}) and the second 
equation of (\ref{hess3}) in (\ref{quasiEinstein}) we get (\ref{eqphii}).
Now for $X_{i}\in{\cal L}(\mathbb{R}^{n})$ and $Y_{j}\in{\cal L}(F)$ ($1\leq i \leq n$ and $1\leq j\leq m$) we get
\[Hess_{\widetilde{g}}h(X_{i}, Y_{j}) = 0.\] In this case  equation (\ref{quasiEinstein}) is trivially satisfied.

Taking $Y_{i}, Y_{j}\in{\cal L}(F)$ with $1\leq i,j\leq m$ and using the last equation of (\ref{ric1}) and the equation (\ref{quasiEinstein}), we have
$$Ric_{g_F}(Y_i,Y_j)=[\rho f^2 + f\triangle_{\overline{g}}f + (m-1)|grad_{\overline{g}}f|^2]g_F(Y_i,Y_j)+\frac{r}{h}Hess_{\widetilde{g}}h(Y_i,Y_j).$$
Being $F$ Einstein, we should have
$$\frac{r}{h}Hess_{\widetilde{g}}h(Y_i,Y_j)=\mu g_{F}(Y_i,Y_j).$$
But
$$\frac{r}{h}Hess_{\widetilde{g}}h(Y_i,Y_j)=\frac{r}{h}\widetilde{g}(\nabla_{Y_i}(grad_{\widetilde{g}}h),Y_j) = \frac{r}{h}\frac{(grad_{\overline{g}}h)(f)}{f}\widetilde{g}(Y_i,Y_j)$$
\begin{equation}\label{hess4}
\Longleftrightarrow Hess_{\widetilde{g}}h(Y_i,Y_j)=f grad_{\overline{g}}h(f)g_F(Y_i,Y_j) = f\varphi^2\sum\limits_{k=1}^{n}\veps_kf,_{x_k}h,_{x_k}g_F(Y_i,Y_j).
\end{equation}

By substituting (\ref{ric5}) and (\ref{hess4}) into (\ref{quasiEinstein}) we obtain (\ref{eqphll}).

The converse of this theorem can be easily verified.

In the case  $m = 1$ just note that, being $V\in\mathcal{L}(F)$, we have

\begin{eqnarray*}
  Ric_{\widetilde{g}}(X_{i},X_{j}) &=& Ric_{\overline{g}}(X_{i},X_{j}) - \frac{1}{f}Hess_{\overline{g}}f(X_{i},X_{j}),\ \forall\  i,\ j = 1,\ldots n \\
  Ric_{\widetilde{g}}(X_{i},V) &=& 0,\ \forall\ i= 1,\ldots n\\
  Ric_{\widetilde{g}}(V, V) &=& - \widetilde{g}(V, V)\frac{\triangle_{\overline{g}}f}{f}.
\end{eqnarray*}
In this case the equations (\ref{eqphij}) and (\ref{eqphii}) remain the same and  equation (\ref{eqphll}) reduces to

$$\sum\limits_{k=1}^{n}\veps_k[-h\varphi^2f,_{x_kx_k}+(n-2)h\varphi\varphi,_{x_k}f,_{x_k}-r\varphi^2f,_{x_k}h,_{x_k}]=\rho f.$$

This concludes the proof of Theorem \ref{teor2}.
\end{proof}

\vspace{.2in}

\begin{proof}[Proof of Theorem \ref{teor3}:]

We are assuming that $f=f(\xi)$, $\varphi=\varphi(\xi)$ and $h=h(\xi)$, where $\xi=\sum\limits_{k=1}^{n}\alpha_kx_k$, with $\alpha_i\in\mathbb{R}$ and $\sum\limits_{k=1}^{n}\varepsilon_k\alpha_k^2=\varepsilon_{i_0}$ or $\sum\limits_{k=1}^{n}\varepsilon_k\alpha_k^2=0$.
So we have to
$$\begin{array}{cc}
f,_{x_i}=f'\alpha_i, & f,_{x_ix_j}=f''\alpha_i\alpha_j,\\
\varphi,_{x_i}=\varphi'\alpha_i, & \varphi,_{x_ix_j}=\varphi''\alpha_i\alpha_j,\\
h,_{x_i}=h'\alpha_i, & h,_{x_ix_j}=h''\alpha_i\alpha_j.
\end{array}$$

Substituting in equation (\ref{eqphij}), we have
$$(n-2)fh\varphi''\alpha_i\alpha_j-rf\varphi h''\alpha_i\alpha_j-mh\varphi f''\alpha_i\alpha_j-2mh\varphi'f'\alpha_i\alpha_j-2rf\varphi'h'\alpha_i\alpha_j=0, \ \ \ \forall i\neq j.$$
If there is any $i\neq j$ such that $\alpha_i\alpha_j\neq 0$, then this equation becomes
\begin{equation} \label{teor3eqI'}
(n-2)fh\varphi''-rf\varphi h''-mh\varphi f''-2mh\varphi'f'-2rf\varphi'h'=0.
\end{equation}
In the same way, considering equation (\ref{eqphii}), we have
$$\begin{array}{c}
\varphi[(n-2)fh\varphi''\alpha_i^2-rf\varphi h''\alpha_i^2-mh\varphi f''\alpha_i^2-2mh\varphi'f'\alpha_i^2-2rf\varphi'h'\alpha_i^2]\\
+\varepsilon_i\sum\limits_{k=1}^{n}\varepsilon_k[fh\varphi\varphi''\alpha_k^2-(n-1)fh(\varphi')^2\alpha_k^2+mh\varphi\varphi'f'\alpha_k^2+rf\varphi\varphi'h'\alpha_k^2]=\varepsilon_i\rho fh.
\end{array}$$
Using equation (\ref{teor3eqI'}), we get
\begin{equation} \label{teor3eqII'}
\sum\limits_{k=1}^{n}\varepsilon_k\alpha_k^2[fh\varphi\varphi''-(n-1)fh(\varphi')^2+mh\varphi\varphi'f'+rf\varphi\varphi'h']=\rho fh.
\end{equation}
Analogously, the equation (\ref{eqphll}) reduces to
\begin{equation} \label{teor3eqIII'}
\sum\limits_{k=1}^{n}\varepsilon_k\alpha_k^2[-fh\varphi^2f''+(n-2)fh\varphi\varphi'f'-(m-1)h\varphi^2(f')^2-rf\varphi^2f'h']=h[\rho f^2-\lambda_F].
\end{equation}

Thus, if we consider the case where $\sum\limits_{k=1}^{n}\varepsilon_k\alpha_k^2=\varepsilon_{i_0}$, worth the system
\begin{equation} \label{teor3sistema*}
\begin{cases}
(n-2)fh\varphi''-rf\varphi h''-mh\varphi f''-2mh\varphi'f'-2rf\varphi'h'=0\\
\sum\limits_{k=1}^{n}\varepsilon_k\alpha_k^2[fh\varphi\varphi''-(n-1)fh(\varphi')^2+mh\varphi\varphi'f'+rf\varphi\varphi'h']=\rho fh\\
\sum\limits_{k=1}^{n}\varepsilon_k\alpha_k^2[-fh\varphi^2f''+(n-2)fh\varphi\varphi'f'-(m-1)h\varphi^2(f')^2-rf\varphi^2f'h']=h[\rho f^2-\lambda_F]
\end{cases}
\end{equation}

But, if $\sum\limits_{k=1}^{n}\varepsilon_k\alpha_k^2=0$, we have by (\ref{teor3eqII'}) to $\rho=0$, and then for (\ref{teor3eqIII'}) follow that $\lambda_F=0$. Therefore, in this case we obtain
\begin{equation} \label{teor3sistema**}
\begin{cases}
(n-2)fh\varphi''-rf\varphi h''-mh\varphi f''-2mh\varphi'f'-2rf\varphi'h'=0\\
\rho=\lambda_F=0
\end{cases}
\end{equation}

Now, if $\alpha_i\alpha_j=0$ for all $i\neq j$, then $\xi=x_{i_0}$. Therefore, the equation (\ref{eqphij}) is trivially satisfied. For the other equations, we will think of two cases:
\begin{itemize}
\item $i\neq i_0$\\
In this case, as $\alpha_i=0$ for all $i\neq i_0$, the equation (\ref{eqphii}) becomes
$$\varepsilon_{i_0}\alpha_{i_0}^2[fh\varphi\varphi''-(n-1)fh(\varphi')^2+mh\varphi\varphi'f'+rf\varphi\varphi'h']=\rho fh,$$
and the equation (\ref{eqphll}) stay
$$\varepsilon_{i_0}\alpha_{i_0}^2[-fh\varphi^2f''+(n-2)fh\varphi\varphi'f'-(m-1)h\varphi^2(f')^2-rf\varphi^2f'h']=h[\rho f^2-\lambda_F].$$

\item $i=i_0$\\
In this case, the equation (\ref{eqphii}) provides
$$\begin{array}{c}
\alpha_{i_0}^2\varphi[(n-2)fh\varphi''-rf\varphi h''-mh\varphi f''-2mh\varphi'f'-2rf\varphi'h']\\
+ \varepsilon_{i_0}^2\alpha_{i_0}^2[fh\varphi\varphi''-(n-1)fh(\varphi')^2+mh\varphi\varphi'f'+rf\varphi\varphi'h']=\varepsilon_{i_0}\rho fh.
\end{array}$$
But for the previous case, we must
$$\varepsilon_{i_0}^2\alpha_{i_0}^2[fh\varphi\varphi''-(n-1)fh(\varphi')^2+mh\varphi\varphi'f'+rf\varphi\varphi'h']=\varepsilon_{i_0}\rho fh.$$
Therefore, we obtain the first equation of (\ref{teor3sistemavetorNAOnulo}). The equation (\ref{eqphll}) continues being
$$\varepsilon_{i_0}\alpha_{i_0}^2[-fh\varphi^2f''+(n-2)fh\varphi\varphi'f'-(m-1)h\varphi^2(f')^2-rf\varphi^2f'h']=h[\rho f^2-\lambda_F].$$
\end{itemize}
Then we conclude that the system (\ref{teor3sistema*}) prevails. This concludes the proof of Theorem \ref{teor3}.
\end{proof}

\vspace{.2in}

In order to prove Theorems \ref{teor4} and \ref{teor5}, we consider functions $f(\xi)$, $h(\xi)$ and $\varphi(\xi)$, where $\xi=\sum\limits_{i=1}^{n}\alpha_ix_i$, $\alpha_i\in\mathbb{R}$, $\sum\limits_{i=1}^{n}\varepsilon_i\alpha_i^2=\pm 1$. It follows from Theorem \ref{teor3} that $(M=\mathbb{R}^n\times_fF,\widetilde{g} = \overline{g} + f^{2}g_{F},h)$, $\displaystyle \overline{g}=\frac{1}{\varphi^2}g$, is a quasi-Einstein manifold if, only if, the functions $f$, $\varphi$ and $h$ satisty
$$\begin{cases}
(n-2)fh\varphi''-rf\varphi h''-mh\varphi f''-2mh\varphi'f'-2rf\varphi'h'=0\\
fh\varphi\varphi''-(n-1)fh(\varphi')^2+mh\varphi\varphi'f'+rf\varphi\varphi'h'=0\\
-fh\varphi^2f''+(n-2)fh\varphi\varphi'f'-(m-1)h\varphi^2(f')^2-rf\varphi^2f'h'=0
\end{cases},$$

i.e.,

\begin{equation} \label{sistema**}
\begin{cases}
(n-2)\frac{\varphi''}{\varphi}-r\frac{h''}{h}-m\frac{f''}{f}-2m\frac{\varphi'}{\varphi}\frac{f'}{f}-2r\frac{\varphi'}{\varphi}\frac{h'}{h}=0\\
\frac{\varphi''}{\varphi}-(n-1)\left(\frac{\varphi'}{\varphi}\right)^2+m\frac{\varphi'}{\varphi}\frac{f'}{f}+r\frac{\varphi'}{\varphi}\frac{h'}{h}=0\\
-\frac{f''}{f}+(n-2)\frac{\varphi'}{\varphi}\frac{f'}{f}-(m-1)\left(\frac{f'}{f}\right)^2-r\frac{f'}{f}\frac{h'}{h}=0
\end{cases}.
\end{equation}

Note that
$$\frac{\varphi''}{\varphi}=\frac{\varphi\varphi''}{\varphi^2}-\left(\frac{\varphi'}{\varphi}\right)^2+\left(\frac{\varphi'}{\varphi}\right)^2=\left(\frac{\varphi'}{\varphi}\right)'+\left(\frac{\varphi'}{\varphi}\right)^2.$$

Thus, the system (\ref{sistema**}) is equivalent to
$$\begin{cases}
(n-2)\left(\frac{\varphi'}{\varphi}\right)'+(n-2)\left(\frac{\varphi'}{\varphi}\right)^2-r\left(\frac{h'}{h}\right)'-r\left(\frac{h'}{h}\right)^2-m\left(\frac{f'}{f}\right)'-m\left(\frac{f'}{f}\right)^2-2m\frac{\varphi'}{\varphi}\frac{f'}{f}-2r\frac{\varphi'}{\varphi}\frac{h'}{h}=0\\
\left(\frac{\varphi'}{\varphi}\right)'-(n-2)\left(\frac{\varphi'}{\varphi}\right)^2+m\frac{\varphi'}{\varphi}\frac{f'}{f}+r\frac{\varphi'}{\varphi}\frac{h'}{h}=0\\
-\left(\frac{f'}{f}\right)'-m\left(\frac{f'}{f}\right)^2+(n-2)\frac{\varphi'}{\varphi}\frac{f'}{f}-r\frac{f'}{f}\frac{h'}{h}=0
\end{cases}.$$

It follows from the second equation
$$\frac{h'}{h}=-\frac{1}{r}\frac{\left(\frac{\varphi'}{\varphi}\right)'}{\left(\frac{\varphi'}{\varphi}\right)}+\frac{(n-2)}{r}\frac{\varphi'}{\varphi}-\frac{m}{r}\frac{f'}{f}.$$

On the other hand, by the third equation we have to
$$\frac{h'}{h}=-\frac{1}{r}\frac{\left(\frac{f'}{f}\right)'}{\left(\frac{f'}{f}\right)}+\frac{(n-2)}{r}\frac{\varphi'}{\varphi}-\frac{m}{r}\frac{f'}{f}.$$

Therefore,
$$\frac{\left(\frac{\varphi'}{\varphi}\right)'}{\left(\frac{\varphi'}{\varphi}\right)}=\frac{\left(\frac{f'}{f}\right)'}{\left(\frac{f'}{f}\right)}.$$

Integrating, we obtain
$$\frac{\varphi'}{\varphi}=k\frac{f'}{f}, \ \ \ \ \ k\in\mathbb{R}^*_+.$$

Substituting into the system, the second and third equations become the same, and then we get
$$\begin{cases}
[(n-2)k-m]\left(\frac{f'}{f}\right)'-r\left(\frac{h'}{h}\right)'-[m(2k+1)-(n-2)k^2]\left(\frac{f'}{f}\right)^2-2rk\frac{f'}{f}\frac{h'}{h}-r\left(\frac{h'}{h}\right)^2=0\\
\left(\frac{f'}{f}\right)'-[(n-2)k-m]\left(\frac{f'}{f}\right)^2+r\frac{f'}{f}\frac{h'}{h}=0
\end{cases}.$$

Denoting by $a=(n-2)k-m$, $b=m(2k+1)-(n-2)k^2$, $x(\xi)=\frac{f'}{f}(\xi)$ and $y(\xi)=\frac{h'}{h}(\xi)$, the system becomes
$$\begin{cases}
ax'-ry'-bx^2-ry^2-2rkxy-ry^2=0\\
x'-ax^2+rxy=0
\end{cases},$$
this is,
\begin{equation} \label{sistemaI}
\begin{cases}
y'=\frac{(a^2-b)}{r}x^2-y^2-(2k+a)xy\\
x'=ax^2-rxy
\end{cases}
\end{equation}

It follows from \ref{sistemaI} that
$$(ax^2-rxy)dy+\left[\left(\frac{b-a^2}{r}\right)x^2+y^2+(2k+a)xy\right]dx=0,$$
where we note that the functions that accompany the terms $dx$ and $dy$ are homogeneous of degree $2$. In order to study this equation (see e.g. \cite{Davis} p. 37), we take
\begin{equation} \label{y=xz}
y(\xi) = x(\xi)z(\xi),
\end{equation}

where z may be a nonzero constant or a nonconstant function. The Theorems \ref{teor4} and \ref{teor5} are obtained by letting z be a nonzero constant when $r\neq 1$ e $r=1$, respectively. Already in Theorem 6 we consider the case where $z$ is a non-constant smooth function.

\vspace{.2in}

\begin{proof}[Proof of Theorem \ref{teor4}:]

We consider solutions of the system as in (\ref{y=xz}), where $z(\xi) = N$ and $N$ is a nonzero constant; i.e., $y(\xi) = Nx(\xi)$. By substituting $y$ in the first and second equations of (\ref{sistemaI}), we get
\begin{equation} \label{sistemaII}
\begin{cases}
x'=\left[\frac{(a^2-b)-N^2r-(2k+a)Nr}{Nr}\right]x^2\\
x'=[a-Nr]x^2
\end{cases}.
\end{equation}

Comparing the two expressions we conclude that $N$ must satisfy
$$r(r-1)N^2-2r(k+a)N+a^2-b=0.$$
Therefore, we get two values of N, given by
\begin{equation} \label{teor4eq42}
N_\pm=\frac{r(k+a)\pm\sqrt{r^2(k+a)^2-r(r-1)(a^2-b)}}{r(r-1)}, \ \ \ \ \ r(k+a)^2\geq (r-1)(a^2-b).
\end{equation}

Going back to the second equation, we have
$$\frac{x'}{x^2}=a-rN_\pm.$$
Consequently,
$$x_\pm(\xi)=-\frac{1}{(a-rN_\pm)\xi+c} \ \ \ \ \ and \ \ \ \ \ y_\pm(\xi)=-\frac{N_\pm}{(a-rN_\pm)\xi+c},$$
where $c\in\mathbb{R}$. Remembering that
$$\frac{\varphi'}{\varphi}=k\frac{f'}{f}, \ \ \ \ \ x(\xi)=\frac{f'}{f}(\xi) \ \ \ \ \ and \ \ \ \ \ y(\xi)=\frac{h'}{h}(\xi),$$
we get $f$, $h$ and $\varphi$ given by \ref{teor4caradasfuncoes}. This concludes the proof of Theorem \ref{teor4}.
\end{proof}

\vspace{.2in}

\begin{proof}[Proof of Theorem \ref{teor5}:]

We consider solutions of the system (\ref{sistemaI}) to $r = 1$ as $y(\xi) = Nx(\xi)$. By substituting $y$ in the first and second equations of (\ref{sistemaI}), we get
\begin{equation} \label{sistemaII'}
\begin{cases}
x'=\left[\frac{(a^2-b)-N^2-(2k+a)N}{N}\right]x^2\\
x'=[a-N]x^2
\end{cases}.
\end{equation}

Comparing the two expressions we conclude that $N$ is given by
\begin{equation} \label{teor5eq42}
N=\frac{a^2-b}{2(k+a)}, \ \ \ \ \ k+a\neq 0.
\end{equation}

Going back to the second equation, we have
$$\frac{x'}{x^2}=a-N.$$
Consequently,
$$x(\xi)=-\frac{1}{(a-N)\xi+c} \ \ \ \ \ and \ \ \ \ \ y(\xi)=-\frac{N}{(a-N)\xi+c},$$
where $c\in\mathbb{R}$. Remembering that
$$\frac{\varphi'}{\varphi}=k\frac{f'}{f}, \ \ \ \ \ x(\xi)=\frac{f'}{f}(\xi) \ \ \ \ \ and \ \ \ \ \ y(\xi)=\frac{h'}{h}(\xi),$$
we get $f$, $h$ and $\varphi$ given by \ref{teor5caradasfuncoes}.

This concludes the proof of Theorem \ref{teor5}.
\end{proof}

\vspace{.2in}

\begin{proof}[Proof of Theorem \ref{teortudoimplicito}:]

We consider solutions of the system (\ref{sistemaI}) as in (\ref*{y=xz}), where $z(\xi)$ is a non-constant smooth function. By substituting $y(\xi)=x(\xi)z(\xi)$ in the first and second equations of (\ref{sistemaI}), we get
\begin{equation} \label{teortudoimplicitosistema}
\begin{cases}
z'=\frac{a^2-b}{r}x-2(k+a)xz+(r-1)xz^2\\
x'=(a-rz)x^2
\end{cases}.
\end{equation}

Dividing the second equation for the first one, we obtain
$$\frac{dx}{x}=\frac{r(a-rz)}{r(r-1)z^2-2r(k+a)z+a^2-b}dz,$$
which shows us that
$$|x|=c_1e^{\int\frac{r(a-rz)}{r(r-1)z^2-2r(k+a)z+a^2-b}dz}, \ \ \ \ \ c_1>0.$$

Denoting by $\displaystyle v(z)=\frac{r(a-rz)}{r(r-1)z^2-2r(k+a)z+a^2-b}$ and returning to the second equation of the system (\ref{teortudoimplicitosistema}), we see that if
$$x=ce^{\int v(z)dz}, \ \ \ \ \ c\in\mathbb{R},$$
then $z$ satisfies
$$v(z)z'-c(a-rz)e^{\int v(z)dz}=0.$$
This concludes the proof of Theorem \ref{teortudoimplicito}.
\end{proof}

\vspace{.2in}

\begin{proof}[Proof of Theorem \ref{teorseesomentese}:]
When $\rho=0$, by introducing the auxiliary functions $\displaystyle x(\xi)=\frac{f'}{f}(\xi)$ and $\displaystyle y(\xi)=\frac{h'}{h}(\xi)$, we have seen that (\ref{teor3sistemavetorNAOnulo}) is equivalent to the system (\ref{sistemaI}) for $x$ and $y$. The solutions of this system can be written as $y(\xi)=x(\xi)z(\xi)$, where $z(\xi)$ is a nonzero function.
If $z(\xi)$ is a nonzero constant, then the proof of Theorems \ref{teor4} and \ref{teor5} shows that the solutions of (\ref{sistemaI}) are given by (\ref{teor4caradasfuncoes}) and (\ref{teor5caradasfuncoes}), respectively. If the function $z(\xi)$ is not constant, then the proof of Theorem \ref{teortudoimplicito} shows that $x$ and $z$ are determined by (\ref{teortudoimplicitoeqs1}) or (\ref{teortudoimplicitoeqs2}) and the functions $f$, $\varphi$ and $h$ are obtained by integrating the ordinary differential equations given by (\ref{teortudoimplicitofhvarphi}). This completes the proof of Theorem \ref{teorseesomentese}.
\end{proof}

\vspace{.2in}

\begin{proof}[Proof of Theorem \ref{teor6}:]
Let $(\R^n,g)$ be a pseudo-Euclidean space, $n\geq 3$ with coordinates $x=(x_1,\cdots, x_n)$ and $g_{ij}=\delta_{ij}\veps_i$. Consider a warped product $M=(\mathbb{R}^n,\overline{g})\times_fF^m$ with metric $\widetilde{g}=\overline{g}\oplus f^2g_F$, where $\overline{g}=\frac{1}{\varphi^2}g$ and F is a Ricci-flat semi-Riemannian manifold. Let $f(\xi)$ and $\varphi(\xi)$ be any positive differentiable functions invariant under the translation of $(n−1)$-dimensional translation group, whose basic invariant is $\xi=\sum\limits_{i=1}^{n}\alpha_ix_i$, with $\alpha_i\in\mathbb{R}$ and $\sum\limits_{i=1}^{n}\varepsilon_i\alpha_i^2=0$. Then it follows from Theorem \ref{teor3} that $(M,\widetilde{g},h)$ is a quasi-Einstein if, and only if, $\rho=\lambda_F=0$ and and $h$ satisfies the linear ordinary differential equation (\ref{teor6equacao}) determined by $f$ and $\varphi$.
\end{proof}

\vspace{.2in}

In order to prove Corollary \ref{cor2}, we consider functions $f(\xi)$, $h(\xi)$ and $\varphi(\xi)$. For the work of Dong-Soo Kim and Young Ho Kim (cf. \cite{Kim-Kim}), if a manifold $(B,g_B,h)$ is quasi-Einstein, this is,
$$Ric_{g_B} - \frac{r}{h}Hess_{g_B}h = \rho g_B, \ \ \ \ \ r\in\mathbb{Z}^*_+, \ \ \ \ \ \rho\in\mathbb{R},$$
then the function $h$ satisfies the equation
\begin{equation} \label{eq3doKimKim}
h\Delta_{g_B}h + (r-1)|grad_{g_B}h|^2 + \rho h^2 = \mu,
\end{equation}
for some constant $\mu\in\mathbb{R}$. Thus, choosing a manifold $F_2$ of dimension $r$ and with Ricci curvature $\mu$, we built a Einstein manifold $B\times_hF_2$.

Let's compute the expression on the left side of the equation (\ref{eq3doKimKim}) in the warped product metric $g_B$. Recalling that in $B=(\mathbb{R}^{n}\times _{f}F_1^{m},g_B)$ the functions $f$, $h$ and $\varphi$ depend only on $(\mathbb{R}^n,\overline{g}=\frac{1}{\varphi^2}g)$, and
$$h,_{x_i}=\alpha_ih', \ \ \ \ \ h,_{x_ix_i}=\alpha_i^2h'' \ \ \ \ e \ \ \ \ \sum\limits_{i=1}^{n}\alpha_i^2\varepsilon_i=\varepsilon_{i_0},$$
and analogously to $f$ and $\varphi$, we have
$$grad_{g_B}f=grad_{\overline{g}}f = \sum\limits_{i=1}^{n}\varphi^2\varepsilon_i\alpha_if'\frac{\partial}{\partial x_i}, \ \ \ \ \ grad_{g_B}h= \sum\limits_{i=1}^{n}\varphi^2\varepsilon_i\alpha_ih'\frac{\partial}{\partial x_i},$$
and
$$\Delta_{\overline{g}}h = \sum\limits_{i,j=1}^{n}\overline{g}^{ij}\left[\frac{\partial^2h}{\partial x_i\partial x_j}-\sum\limits_{k=1}^{n}\overline{\Gamma}_{ij}^k\frac{\partial h}{\partial x_k}\right]
=\varphi^2\varepsilon_{i_0}\left[h''-(n-2)\frac{\varphi'h'}{\varphi}\right].$$

Therefore,
\begin{equation} \label{laplacianodeh}
\begin{array}{rl}
\Delta_{g_B}h & =\Delta_{\overline{g}}h + m\frac{\overline{g}(grad_{\overline{g}}h,grad_{\overline{g}}f)}{f}\\
 & = \varphi^2\varepsilon_{i_0}\left[h''-(n-2)\frac{\varphi'h'}{\varphi}\right] + m\varepsilon_{i_0}\varphi^2\frac{h'f'}{f},
\end{array}
\end{equation}
and
\begin{equation} \label{normadogradientedeh}
|grad_{g_B}h|^2=|grad_{\overline{g}}h|^2=\varepsilon_{i_0}\varphi^2(h')^2.
\end{equation}

Then, in our case, the left side of (\ref{eq3doKimKim}) becomes
\begin{equation} \label{eq3doKimKimnonossocaso}
\begin{array}{l}
h\Delta_{g_B}h + (r-1)|grad_{g_B}h|^2 + \rho h^2\\
= h\Big\{\varphi^2\varepsilon_{i_0}\left[h''-(n-2)\frac{\varphi'h'}{\varphi}\right] + m\varepsilon_{i_0}\varphi^2\frac{h'f'}{f}\Big\} + (r-1)\varepsilon_{i_0}\varphi^2(h')^2 + \rho h^2.
\end{array}
\end{equation}

\vspace{.2in}

\begin{proof}[Proof of Corollary \ref{cor2}:]
When the vector $\alpha=\sum_{i=1}^{n}\varepsilon_i\alpha_i^2$ is non null, we divide the proof in two cases:

\begin{itemize}
\item[i)] $r\neq 1$\\
Consider $\rho=0$, $r\neq 1$ and the functions $f$, $h$ and $\varphi$ given by (\ref{teor4caradasfuncoes}). By \ref{eq3doKimKimnonossocaso} we have
$$h\Delta_{g_B}h + (r-1)|grad_{g_B}h|^2 + \rho h^2$$
$$= h_\pm\Big\{\varphi_\pm^2\varepsilon_{i_0}\left[h_\pm''-(n-2)\frac{\varphi_\pm'h_\pm'}{\varphi_\pm}\right] + m\varepsilon_{i_0}\varphi_\pm^2\frac{h_\pm'f_\pm'}{f_\pm}\Big\} + (r-1)\varepsilon_{i_0}\varphi_\pm^2(h_\pm')^2$$
$$=\varepsilon_{i_0}c_2^2c_3^2N_\pm[N_\pm(1-r)+a-(n-2)k+m-(1-r)N_\pm][(a-rN_\pm)\xi+c]^{\frac{2N_\pm(r-1)-2a-2k}{a-rN_\pm}}=0,$$
because $a=(n-2)k-m$. Then, by the result in \cite{Kim-Kim}, if we consider a Einstein manifold $F_2$ of dimension $r$ and with Ricci curvature $\mu=0$, the manifold $(\mathbb{R}^n\times_fF_1^m)\times_hF_2^r$ with metric $\widetilde{g}=(\overline{g}\oplus f^2g_{F_1})\oplus h^2g_{F_2}$ is a Ricci-flat Einstein manifold.

\item[ii)] $r=1$\\
Consider $\rho=0$, $r=1$ and the functions $f$, $h$ and $\varphi$ given by (\ref{teor5caradasfuncoes}). By \ref{eq3doKimKimnonossocaso} we have
$$h\Delta_{g_B}h + (r-1)|grad_{g_B}h|^2 + \rho h^2 = h\Big\{\varphi^2\varepsilon_{i_0}\left[h''-(n-2)\frac{\varphi'h'}{\varphi}\right] + m\varepsilon_{i_0}\varphi^2\frac{h'f'}{f}\Big\}$$
$$=\varepsilon_{i_0}c_2^2c_3^2N[a-(n-2)k+m][(a-N)\xi+c]^{-\frac{2(k+a)}{a-N}}=0,$$
because $a=(n-2)k-m$. By the result in \cite{Kim-Kim}, the manifold $(\mathbb{R}^n\times_fF_1^m)\times_h\mathbb{R}$ with metric $\widetilde{g}=(\overline{g}\oplus f^2g_{F_1})\oplus h^2g_{\mathbb{R}}$ is a Ricci-flat Einstein manifold.
\end{itemize}

Now, when the vector $\alpha=\sum_{i=1}^{n}\varepsilon_i\alpha_i^2$ is null, we have to $\varepsilon_{i_0}=\sum_{i=1}^{n}\varepsilon_i\alpha_i^2=0$. Remembering that $\rho=0$, follows of the equation (\ref{eq3doKimKimnonossocaso}) that the fiber $F_2$ must have null Ricci curvature. Then, by the result in \cite{Kim-Kim}, if we consider a Einstein manifold $F_2$ of dimension $r$ and with Ricci curvature $\mu=0$, the manifold $(\mathbb{R}^n\times_fF_1^m)\times_hF_2^r$ with metric $\widetilde{g}=(\overline{g}\oplus f^2g_{F_1})\oplus h^2g_{F_2}$ is a Ricci-flat Einstein manifold.
\end{proof}

\vspace{.2in}

\begin{proof}[Proof of Corollary \ref{cor3}:]
In the Theorem \ref{teortudoimplicito} we obtained quasi-Einstein manifolds $(B,g_B,h)$ with $\rho=0$, where $B=\mathbb{R}^n\times_fF_1^m$ and $g_B=\overline{g}\oplus f^2g_{F_1}$, and the functions are implicitly given by (\ref{teortudoimplicitofhvarphi}). Then,
$$Ric_{g_B} - \frac{r}{h}Hess_{g_B}h=\rho g_B=0, \ \ \ \ \ r>0.$$
By the work of Dong-Soo Kim and Young Ho Kim (cf. \cite{Kim-Kim}), the function $h$ satisfies the equation (\ref{eq3doKimKim}) for some constant $\mu\in\mathbb{R}$. Choosing a manifold $F_2$ of dimension $r$ and with Ricci curvature $\mu$, we built a Ricci-flat Einstein manifold $(\mathbb{R}^n\times_fF_1^m)\times_hF_2^r$ with metric $\widetilde{g}=(\overline{g}\oplus f^2g_{F_1})\oplus h^2g_{F_2}$.
\end{proof}

\end{document}